\documentclass[a4paper,12pt]{article}
\usepackage{amssymb}
\usepackage{latexsym}
\usepackage{amsmath,amsthm,amssymb}
\usepackage{amsfonts}
\usepackage{eufrak}

\newtheorem{thm}{Theorem}[section]

\newtheorem{pro}[thm]{Proposition}

\theoremstyle{definition}

\newtheorem{exa}[thm]{Example}
\newtheorem{rem}[thm]{Remark}

\begin{document}

\begin{center}
{\Large A variational formula on the Cram\'er function of series of independent random variables}
\end{center}
\begin{center}
{\sc Krzysztof Zajkowski}\\
Institute of Mathematics, University of Bialystok \\ 
Akademicka 2, 15-267 Bialystok, Poland \\ 
kryza@math.uwb.edu.pl 
\end{center}

\begin{abstract}
In \cite{Zaj1} it has been proved some variational formula on the Legendre-Fenchel transform of the cumulant generating function (the Cram\'er function) of Rademacher series 
with coefficients in the space $\ell^1$. In this paper we show a generalization of this formula to series of a larger class of any independent random variables with coefficients that belong to the space $\ell^2$.
\end{abstract}

{\it 2010 Mathematics Subject Classification:} 44A15 

{\it Key words: Cram\'er function, Legendre-Fenchel transform, contraction principle (large deviations theory), rate function, infimal convolution} 

\section{Introduction}

The {\it Legendre-Fenchel transforms} of {\it cumulant generating functions} of  given random variables  are at the core of the {\it large deviations theory} (see e.g. \cite{DemZei,Holl}). The Cram\'er function gives the rate of the exponentially decay of tails of distributions for the empirical means of sequences of i.i.d. random variables. It provides a nice connection between convex analysis and statistics.

Donsker and Varadhan in \cite{DiV} 
proved  that the Legendre-Fenchel transform $\psi_X^\ast$ of the cumulant generating function $\psi_X(s)=\ln Ee^{sX}$ of a random variable $X$  satisfies the following variational principle
\begin{equation}
\label{relE}
\psi_X^\ast(a)=\inf_{m\ll \mu_X,\;\int x dm=a}D(m\Vert \mu_X), 
\end{equation}
where $D(m\Vert \mu_X)=\int\ln\frac{dm}{d\mu_X}dm$ is the {\it relative entropy} of a probability distribution $m$ with respect to the distribution $\mu_X$ of $X$.

The aim of this paper is to prove some variational formula for the Cram\'er functions of series of independent random variables  that depends on coefficients and Cram\'er functions of summands of a given series;
see Theorem \ref{mtw}.

For series $\sum t_ig_i$ of independent standard normal r.v.s, where $\sum t_i^2<\infty$ and $g_i\in\mathcal{N}(0,1)$, it is known their tail estimations of the form
$$
Pr\Big(\sum t_ig_i > \alpha\Big)\le \exp\Big(-\frac{\alpha^2}{2\sum t_i^2}\Big);
$$
see for instance \cite{Tal2}.
The function $\frac{\alpha^2}{2\sum t_i^2}$ is the  Cram\'er function of the random series $\sum t_ig_i$ (see Ex.\ref{sect2})).

To realize  our purposes we will need the general notion of the Legendre-Fenchel transform  in topological spaces (see \cite{EkeTem} or \cite{BarPre}).
Let $V$ be a real locally convex Hausdorff space and $V^*$ its dual space. By $\left\langle\cdot,\cdot\right\rangle$ we denote the canonical pairing between $V$ and $V^*$. Let
$f:V\mapsto \mathbb{R}\cup\{\infty\}$ be a function nonidentically $\infty$. 
By $\mathcal{D}(f)$ we denote   the {\it effective domain} of $f$, i.e. $\mathcal{D}(f)=\{u\in V:\;f(u)<\infty\}$.
 A function $f^\ast:V^\ast\mapsto \mathbb{R}\cup\{\infty\}$  defined by
$$
f^\ast(u^*)=\sup_{u\in V}\{\left\langle u,u^*\right\rangle-f(u)\}=\sup_{u\in \mathcal{D}(f)}\{\left\langle u,u^*\right\rangle-f(u)\}\;\;\;\;\;(u^*\in V^\ast)
$$
is called   the {\it Legendre-Fenchel transform} ({\it convex conjugate}) of  $f$ and   
a function $f^{\ast\ast}:V\mapsto \mathbb{R}\cup\{\infty\}$ defined by 
$$
f^{\ast\ast}(u)=\sup_{u^\ast\in V^\ast}\{\left\langle u,u^\ast\right\rangle-f^\ast(u^\ast)\}=\sup_{u^\ast\in \mathcal{D}(f^\ast)}\{\left\langle u,u^*\right\rangle-f^\ast(u^\ast)\}\;\;\;\;\;(u\in V)
$$
is called the {\it convex biconjugate} of $f$.

The functions $f^\ast$ and $f^{\ast\ast}$ are convex and lower semicontinuous in the weak* and weak topology on $V^\ast$ and $V$, respectively. Moreover,
the  {\it biconjugate theorem} states that the function $f:V\mapsto \mathbb{R}\cup\{\infty\}$  not identically equal to $+\infty$ is convex and lower semicontinuous if and only if 
$f=f^{\ast\ast}$.

Let us mention additional properties of the convex conjugates; see 4.3 Examples in \cite{EkeTem}. Let $V$ be a normed space. We denote by $\Vert\cdot\Vert$ the norm of $V$ and by $\Vert\cdot\Vert_\ast$ the norm of $V^\ast$. For conjugates exponents $p,q\in(1,\infty)$ ($\frac{1}{p}+\frac{1}{q}=1$), a function
$\frac{1}{q}\Vert u^\ast \Vert_\ast^q$
is the convex conjugate of $\frac{1}{p}\Vert u \Vert^p$. 
\begin{rem}
\label{rem1}
Let us emphasize that in Hilbert spaces a function $\frac{1}{2}\Vert u \Vert^2$
one can treat as  the function invariant with respect to the Legendre-Fenchel transform.
\end{rem}
Let us list two properties and the notion of the infimal convolution. The convex-conjugation is order-reversing:
\begin{equation}
\label{orev}
{\rm if} \;f\le g\; {\rm then}\; f^\ast\ge g^\ast
\end{equation} 
and 
\begin{equation}
\label{jedn}
{\rm if}\; g(u)=f(au)\; {\rm where}\; (a\neq 0)\; {\rm then} \;g^\ast(u^\ast)=f^\ast\Big(\frac{u^\ast}{a}\Big).
\end{equation}
Let functions $f_1,...,f_n$ are convex, lower semicontinuous and not identically equal to $+\infty$. Suppose there is a point in $\bigcap_{i=1}^n\mathcal{D}(f_i)$ at which $f_1,...,f_{n-1}$ are continuous. Then   the convex conjugate of their sum is given by the so called infimal (in this case even minimal) convolution, i.e.
$$
(f_1+...+f_n)^\ast(u^\ast)=\min_{u^\ast_1+...+u^\ast_n=u^\ast}\{f_1^\ast(u^\ast_1)+...+f_n^\ast(u^\ast_n)\}
$$ 
(see e.g. \cite[Th.1]{HiU}).

It will turn out (see Rem.\ref{uw}) that the variational formula for the Cram\'er function of series of independent random variables is an example of an application of a generalization of the infimal convolution to the case of infinite numbers of summands. Often generalizations of formulas from finite numbers parameters (variables) to the case of infinite ones are not obvious. Other examples of a generalizations of the convex conjugates of the logarithm of  series of analytic functions, 
with applications to investigations of the convex conjugates of the spectral radius of the functions of weighted composition operators, 
one can find in \cite{OZ3,Zaj2}.

\section{Main Theorem}
\label{sect3}
The cumulant generating function $\psi_X(s)=\ln Ee^{sX}$ of any random variable $X$ is convex and lower semicontinuous on $\mathbb{R}$ (analytic on $int \mathcal{D}(\psi_X)$). It maps $\mathbb{R}$ into $\mathbb{R}\cup \{\infty\}$ and takes value zero at zero but it is possible that $\psi_X(s)=\infty$ when $s\neq 0$. We will assume that it is finite on some neighborhood of zero, i.e. $X$  satisfies condition: $\exists_{\lambda>0}$ s.t. $Ee^{\lambda|X|}<\infty$.
Let us emphasize that if $EX=0$ then $\psi_X\ge 0$ but the Cram\'er transform $\psi_X^\ast$ is always nonnegative and attains $0$ at the value $EX$.

Let $I\subset \mathbb{N}$ and  $(X_i)_{i\in I}$ be a sequence of independent r.v.s. For ${\bf t}=(t_i)_{i\in I}\in\ell^2(I)\equiv\ell^2$
consider $X_{\bf t}=\sum_{i\in I}t_i X_i$ (convergence in $L^2$ and almost surely). Observe that $EX_{\bf t}^2=\sum_{i\in I}t_i^2=\Vert {\bf t} \Vert^2$. 
The cumulant generating function of $X_{\bf t}$ we will denote by $\psi_{\bf t}$, i.e. $\psi_{\bf t}=\psi_{X_{\bf t}}$. Notice that for fixed $s$ we can consider $\psi_{\bf t}(s)$ as a functional of the variable ${\bf t}$ in $\ell^2$.
We will denote it by $\psi^s$. Let us emphasize that $\psi^s({\bf t})=\psi_{\bf t}(s)$.

Before proving our main theorem, we show forms of the cumulant generating function and the Cram\'er transform of series of independent random variables. 
\begin{pro}
\label{formpsi}
Let $(X_i)_{i\in I}$ be a sequence of zero-mean independent random variables with common bounded second moments. Let for each $i\in I$ the cumulant generating function 
$\psi_i(s):=\psi_{X_i}(s)=\ln Ee^{sX_i}$ is finite on some neighborhood of zero. Then for each ${\bf t}=(t_i)_{i\in I}\in \ell^2$ the cumulant generating function of a random series $X_{\bf t}=\sum_{i\in I}t_iX_i$ is given by the following equality
$$
\psi_{\bf t}(s):=\psi_{X_{\bf t}}(s)=\sum_{i\in I}\psi_i(st_i).
$$
\end{pro}
\begin{proof}
Because $(X_i)$ are independent, centered and have  common bounded second moments then for every ${\bf t}\in\ell^2$ the series $X_{\bf t}=\sum_{i\in I}t_iX_i$ converges in $L^2$ and a.s.. Let us emphasize that the convergence of series $X_{\bf t}$ in $L^2$ is equivalent to the convergence of sequences ${\bf t}$ in $\ell^2$. By fixed $s$ we can consider the cumulant generating function $\psi_{\bf t}(s)$ as a functional of the variable  ${\bf t}$ in $\ell^2$. We will denote it by $\psi^s({\bf t})$, i.e. $\psi^s({\bf t})=\psi_{\bf t}(s)$. We show that for every $s\in\mathbb{R}$ the functional $\psi^s$ is convex and lower semicontinuous on $\ell^2$.

Convexity one may check by using the H\"older inequality. Let ${\bf t},{\bf u}\in \ell^2$ and $\lambda\in(0,1)$
then 
\begin{eqnarray*}
\psi^s(\lambda{\bf t}+(1-\lambda){\bf u}) & =& \ln Ee^{s\sum_{i\in I}(\lambda t_i+(1-\lambda) u_i)X_i}\\
\; &=& \ln E\big[\big(e^{s\sum_{i\in I}t_iX_i}\big)^\lambda\big(e^{s\sum_{i\in I}u_iX_i}\big)^{1-\lambda}\big]
=\ln E\big[\big(e^{sX_{\bf t}}\big)^\lambda\big(e^{sX_{\bf u}}\big)^{1-\lambda}\big].
\end{eqnarray*}
By the H\"older inequality for exponents $1/\lambda$ and $1/(1-\lambda)$ we get
$$
E\big[\big(e^{sX_{\bf t}}\big)^\lambda\big(e^{sX_{\bf u}}\big)^{1-\lambda}\big]
\le \big(Ee^{sX_{\bf t}}\big)^\lambda\big(Ee^{sX_{\bf u}}\big)^{1-\lambda}
$$
and, in consequence,
\begin{eqnarray*}
\psi^s(\lambda{\bf t}+(1-\lambda){\bf u})& \le &\lambda\ln Ee^{sX_{\bf t}} + (1-\lambda)\ln Ee^{sX_{\bf u}}\\
\; & = & \lambda\psi^s({\bf t})+(1-\lambda)\psi^s({\bf u}).
\end{eqnarray*}

Lower semicontinuity follows from Fatou's lemma. Let ${\bf t}^n\to {\bf t}^0$ in $\ell^2$. Note that $X_{{\bf t}^n}$ converges a.s. to $X_{{\bf t}^0}$. Then
\begin{eqnarray*}
\liminf_{n\to\infty}\psi^s({\bf t}^n)  =  \liminf_{n\to\infty}\ln Ee^{sX_{{\bf t}^n}} & \ge & \ln E(\liminf_{n\to\infty}e^{sX_{{\bf t}^n}})\\
\; & = & \ln E(e^{s\lim_{n\to\infty}X_{{\bf t}^n}})=\psi^s({\bf t}^0).
\end{eqnarray*} 
It means that $\psi^s$ is lower semicontinuous on $\ell^2$.

Let $\ell_0$ denote the space of sequences with  finite supports. Observe that $\ell_0$ is a dense subset of $\ell^2$. 
For ${\bf t}\in\ell_0$ we have
\begin{eqnarray*}
\psi^s({\bf t}) & = & \ln Ee^{sX_{\bf t}} = \ln\prod_{i\in I}e^{st_iX_i}\\
\; & = & \sum_{i\in I}\psi_i(st_i).
\end{eqnarray*}
For ${\bf t}\in\ell^2$ consider a series $\sum_{i\in I}\psi_i(st_i)$. Since $EX_i=0$, $\psi_i\ge 0$. It follows that $\sum_{i\in I}\psi_i(st_i)$ is convergent or divergent to plus infinity. Since $\psi_i$ are convex, this series defines a convex function on  the whole $\ell^2$. Let ${\bf t}^n\to {\bf t}^0$ in $\ell^2$. 
Hence for every $i\in I$
$t^n_i\to t^0_i$. By superadditivity of the limit inferior and, next, by lower semicontinuity of each $\psi_i$, we get 
\begin{eqnarray*}
\liminf_{n\to\infty}\psi^s({\bf t}^n)  =  \liminf_{n\to\infty}\sum_{i\in I}\psi_i(st^n_i) & \ge & \sum_{i\in I}\liminf_{n\to\infty}\psi_i(st^n_i)\\
\; & \ge & \sum_{i\in I}\psi_i(st^0_i)=\psi^s({\bf t}^0).
\end{eqnarray*} 
Notice that both functions:  $\psi^s({\bf t})$ and the series $\sum_{i\in I}\psi_i(st_i)$ are convex and lower semicontinuous on $\ell^2$
and moreover coincide on $\ell_0$ (a dense subset of $\ell^2$). It follows that these functions are equal on whole $\ell^2$, i.e.
$$
\psi^s({\bf t})=\sum_{i\in I}\psi_i(st_i)
$$
for every ${\bf t}$ in $\ell^2$.
\end{proof}

Let us observe that for $s=0$ $\psi^0\equiv 0$ and its convex conjugate $(\psi^0)^\ast({\bf a})=0$ for ${\bf a}={\bf 0}$ and $\infty$ otherwise. From now on we assume that $s\neq 0$. A form of $(\psi^s)^\ast$ for $s\neq 0$ is described in the  following:
\begin{pro}
Under the assumptions of Proposition \ref{formpsi}. The convex conjugate of $\psi^s({\bf t})=\sum_{i\in I}\psi_i(st_i)$ defined on $\ell^2$ equals
$$
(\psi^s)^\ast({\bf a})=\sum_{i\in I}\psi_i^\ast\Big(\frac{a_i}{s}\Big)\quad (s\neq 0)
$$
for ${\bf a}\in\ell^2$, where $\psi^\ast_i$'s are the Cram\'er transforms of $X_i$'s.
\end{pro}
\begin{proof}
The convex conjugate $(\psi^s)^\ast$ is convex and lower semicontinuous on $(\ell^2)^\ast\simeq \ell^2$. Assume first that $I$ is a finite set. By virtue of the form of $\psi^s$, the convex conjugate of a separated sum (see e.g. \cite[Prop.13.27]{BaCom}) and the property (\ref{jedn}),
for ${\bf a}$ in $\ell^2(I)$, we get
$$
(\psi^s)^\ast({\bf a})=\sum_{i\in I}\psi_i^\ast\Big(\frac{a_i}{s}\Big)\quad (s\neq 0).
$$
Define now a functional $\sum_{i\in I}\psi_i^\ast(\frac{a_i}{s})$ on whole space $\ell^2$. Since $\psi^\ast_i$'s are convex and lower semicontinuous, this functional is convex and, similarly as in the case of $\sum_{i\in I}\psi_i$, one can show that it is also lower semicontinuous on $\ell^2$. Because this functional coincides with $(\psi^s)^\ast$ on the dense subspace $\ell_0$ then both functionals are equal on $\ell^2$. 
\end{proof}
Let us emphasize that the functions $(\psi^s)^\ast$ are nonnegative and lower semicontinuous. In large deviation theory such functions are called rate functions
(good rate functions when level sets are not only closed but also compact). In the main theorem below we show that the contraction principle applied to the function $(\psi^1)^\ast$ by using a functional  $\left\langle {\bf t}, \cdot\right\rangle$ over $\ell^2$ gives the Cram\'er transform of $X_{\bf t}$.
\begin{thm}
\label{mtw}
Let a sequence  of r.v.'s $(X_i)_{i\in I}$ satisfies the assumptions of Proposition \ref{formpsi}. 
Then for every 
${\bf t}=(t_i)_{i\in I}\in \ell^2$ the Cram\'er transform $\psi_{X_{\bf t}}^\ast= \psi_{\bf t}^\ast$ of a random series $X_{\bf t}=\sum_{i\in I}t_iX_i$ is given by the following variational formula
$$
\psi^\ast_{\bf t}(\alpha)=\inf_{\substack{{\bf b}\in \ell^2:\; \left\langle {\bf t}, {\bf b}\right\rangle=\alpha}}\sum_{i\in I}\psi^\ast_i(b_i),
$$
for $\alpha\in int \mathcal{D}(\psi_{\bf t}^\ast)$, where $\psi^\ast_i$'s are the Cram\'er transform of $X_i$'s.
\end{thm}
\begin{proof}

The functional $\psi^s$ is convex and lower semicontinuous on $\ell^2$. By virtue of the biconjugate theorem we have
$$
\psi^s({\bf t})=\sup_{{\bf a}\in \ell^2}\{\left\langle {\bf t}, {\bf a} \right\rangle-(\psi^s)^\ast({\bf a})\},
$$
where $(\psi^s)^\ast({\bf a})=\sum_{i\in I}(\psi_i)^\ast(\frac{a_i}{s})$ $(s\neq 0)$.
Substituting ${\bf a}=s{\bf b}$ we get 
$$
(\psi^s)^\ast({\bf a})=\sum_{i\in I}\psi_i^\ast\big(\frac{a_i}{s}\big)=\sum_{i\in I}\psi_i^\ast(b_i)=(\psi^1)^\ast({\bf b})
$$
and we can rewrite the above as follows
\begin{equation}
\label{row1}
\psi^s({\bf t})=\sup_{{\bf b}\in \ell^2}\{s\left\langle {\bf t}, {\bf b}\right\rangle-(\psi^1)^\ast({\bf b})\}.
\end{equation}

Let us return to the function $\psi_{\bf t}$ which is convex and lower semicontinuous on $\mathbb{R}$. By the biconjugate theorem we have
$$
\psi_{\bf t}(s)=\sup_{\alpha\in\mathbb{R}}\{s\alpha-\psi_{\bf t}^\ast(\alpha)\}.
$$
Let us recall that $\psi_{\bf t}(s)=\psi^s({\bf t})$. If we split the supremum of (\ref{row1}) into two parts: over $\mathbb{R}$ and hyperplanes $\{{\bf b}\in\ell^2:\; \left\langle {\bf b}, {\bf t}\right\rangle=constant\}$ then we get
\begin{eqnarray*}
\psi_{\bf t}(s) & = & \psi^s({\bf t})=\sup_{\alpha\in\mathbb{R}}\sup_{\substack{{\bf b}\in \ell^2:\\ \left\langle  {\bf t}, {\bf b}\right\rangle=\alpha}}\{s\left\langle  {\bf t}, {\bf b}\right\rangle-(\psi^1)^\ast({\bf b})\}\\
\; & = & \sup_{\alpha\in\mathbb{R}}\{s\alpha-\inf_{\substack{{\bf b}\in \ell^2:\; \left\langle  {\bf t}, {\bf b}\right\rangle=\alpha}}(\psi^1)^\ast({\bf b})\}.
\end{eqnarray*}
Let 
$\varphi_{\bf t}(\alpha)$ denote the function $\inf_{\substack{{\bf b}\in \ell^2:\; \left\langle  {\bf t}, {\bf b}\right\rangle=\alpha}}(\psi^1)^\ast({\bf b})$.
The functional $(\psi^1)^\ast$ is convex on $\ell^2$. Convexity is preserved under contraction by linear transformation (see \cite[Th.III.32]{Holl}). It suffices to state that $\varphi_{\bf t}$ and $\psi_{\bf t}^\ast$ coincide on $int \mathcal{D}(\psi_{\bf t}^\ast)$ (both ones take $\infty$ on a complement of 
$cl\mathcal{D}(\psi_{\bf t}^\ast)$)
that is 
$$
\psi^\ast_{\bf t}(\alpha)=\inf_{\substack{{\bf b}\in \ell^2:\; \left\langle  {\bf t}, {\bf b}\right\rangle=\alpha}}(\psi^1)^\ast({\bf b}),
$$
for $\alpha\in int \mathcal{D}(\psi_{\bf t}^\ast)$, where $(\psi^1)^\ast({\bf b})=\sum_{i\in I}\psi^\ast_i(b_i)$. 
\end{proof}
\begin{rem}
We cannot prove in general that $\varphi_{\bf t}$ is lower semicontinuous. Sometimes it is obvious when for instance the effective domain of $\psi_{\bf t}^\ast$ is open subset of $\mathbb{R}$ (or even whole $\mathbb{R}$; see  Examples \ref{sect2} and \ref{przy1}). Under an assumption that $(\psi^1)^\ast$ is a good rate function
with respect to weak* topology  we can state lower-semicontinuity of $\varphi_{\bf t}$ (see Example \ref{przy2}).
\end{rem}

\begin{exa}
\label{sect2}
The moment generating function of a standard normal r.v. $g\in\mathcal{N}(0,1)$ equals
$
Ee^{sg}=\frac{1}{\sqrt{2\pi}}\int_\mathbb{R}e^{st-\frac{t^2}{2}}dt=e^\frac{s^2}{2}
$
and its cumulant generating function
$$
\psi_g(s):=\ln Ee^{sg}=\frac{s^2}{2}.
$$
The function $\frac{s^2}{2}$ is invariant with respect to Legendre transform (see Remark \ref{rem1}) that is the Cram\'er function of $g$ to be
$$
\psi_g^\ast(\alpha)=\frac{\alpha^2}{2}.
$$
By virtue of Proposition \ref{formpsi} the cumulant generating function of the series $X_{\bf t}=\sum_{i\in I}t_ig_i$, where $g_i$ are independent and standard normal distributed,
equals
$$
\psi_{\bf t}(s)=\sum_{i\in I}\psi_g(st_i)=\frac{1}{2}s^2\sum_{i\in I}t_i^2=\frac{1}{2}s^2\Vert {\bf t} \Vert^2.
$$
By the scaling property (\ref{jedn}) we obtain the evident form of the Cram\'er transform
$$
\psi_{\bf t}^\ast(\alpha)=\frac{\alpha^2}{2\Vert {\bf t} \Vert^2}.
$$
On the other hand by Theorem \ref{mtw} we get
$$
\psi_{\bf t}^\ast(\alpha)=\frac{1}{2}\inf_{\substack{{\bf b}\in \ell^2:\\ \left\langle {\bf t},{\bf b}\right\rangle=\alpha}}\Vert {\bf b}\Vert^2.
$$

Using the Lagrange multipliers technique one can check that this infimum is attained at a sequence ${\bf b}=\Big(\frac{\alpha t_i}{\Vert \bf{t} \Vert^2}\Big)_{i\in I}$.
\end{exa}

\begin{exa}
\label{przy1}
Let $X$ be r.v. with the Laplace density $\frac{1}{2}e^{-|x|}$. Its moment generating function $Ee^{sX}=\frac{1}{1-s^2}$ for $|s|<1$ and $\infty$ otherwise. 
Let us observe that 
$$
Ee^{sX}=\frac{1}{1-s^2}=\sum_{n=0}^\infty s^{2n}\ge\sum_{n=0}^\infty\frac{s^{2n}}{n!2^n}=e^\frac{s^2}{2}=Ee^{sg},
$$
where $g$ is standard normal distributed. Thus $\psi_X\ge \psi_g$ and, since the Legendre-Fenchel transform is order-reversing, we have
\begin{equation}
\label{nier}
\psi_X^\ast(\alpha)\le \psi_g^\ast(\alpha)=\frac{\alpha^2}{2}.
\end{equation} 
Moreover by using the classical Legendre transform we can calculate an evident form of $\psi_X^\ast$ and get
$$
\psi_X^\ast(\alpha)=\frac{\alpha^2}{\sqrt{1+\alpha^2}+1}+\ln\frac{2}{\sqrt{1+\alpha^2}+1}.
$$

Consider a sequence $(X_i)_{i\in I}$ of independent r.v.s with the same Laplace distribution. 
By (\ref{nier}) we have
$$
(\psi^1)^\ast({\bf b})=\sum_{i\in I}\psi^\ast_X(b_i)\le \frac{1}{2}\Vert {\bf b} \Vert^2.
$$
It means that $(\psi^1)^\ast$ takes finite values on the whole space $\ell^2$ and for every $\alpha\in\mathbb{R}$ there is a finite infimum: 
$$
\inf_{\substack{{\bf b}\in \ell^2:\; \left\langle  {\bf t}, {\bf b}\right\rangle=\alpha}}\sum_{i\in I}\Big(\frac{b_i^2}{\sqrt{1+b_i^2}+1}+\ln\frac{2}{\sqrt{1+b_i^2}+1}\Big),
$$
that is the finite value of $\psi_{\bf t}^\ast$ at $\alpha$.
\end{exa}
In the paper \cite{Zaj1} one can find an example of the variational formula for the Cram\'er function of series of weighted symmetric Bernoulli random variables but with coefficients belonging to the space $\ell^1$. In the context of our Theorem \ref{mtw}, we recall the main result of this paper but now with coefficients in the larger space $\ell^2$.

\begin{exa}
\label{przy2}
If $X$ is a symmetric Bernoulli r.v., i.e. $Pr(X=\pm 1)=\frac{1}{2}$, then $Ee^{sX}=\cosh s$. By power series expansions one has $\cosh s\le exp(\frac{s^2}{2})$. In this example, conversely as in previous one, we get that $\psi_X(s)\le \psi_g(s)=\frac{s^2}{2}$ and
\begin{equation}
\label{nierB}
\psi_X^\ast(\alpha)\ge \frac{\alpha^2}{2}.
\end{equation}
One can check that
$$
\psi_X^\ast(\alpha)=\frac{1}{2}[(1+\alpha)\ln(1+\alpha)+(1-\alpha)\ln(1-\alpha)]
$$
for $|\alpha|\le 1$ and $\infty$ otherwise; we take $0\ln 0=0$. Note that $\psi_X^\ast(\pm 1)=\ln 2$.
 
For a sequence of independent Bernoulli r.v.s, by the above inequality, we have 
$$
(\psi^1)^\ast({\bf b})\ge \frac{1}{2}\Vert {\bf b} \Vert^2.
$$ 
Since $(\psi^1)^\ast$ is lower semicontinuous in the weak* topology,  level sets $\{{\bf b}\in\ell^2:\;(\psi^1)^\ast({\bf b})\le c\}$ are weak* closed; for $c<0$ are empty sets. By the above, for each $c\ge 0$  the level set of $(\psi^1)^\ast$ is contained in a closed ball 
$\overline{B}({\bf 0};\sqrt{2c})=\{{\bf b}\in\ell^2:\;\Vert {\bf b}\Vert\le\sqrt{2c}\}$. By virtue of the Banach-Alaoglo theorem (see \cite[Th. 3.15]{Rudin}) balls $\overline{B}({\bf 0};\sqrt{2c})$ are weak* compact (are polar sets of balls $\overline{B}({\bf 0};1/\sqrt{2c})$). It follows that the level sets
$\{{\bf b}\in\ell^2:\;(\psi^1)^\ast({\bf b})\le c\}$ are weak* compact as  closed subsets of compact sets. It means that in this topology $(\psi^1)^\ast$ is a good rate function and by the contraction principle (see \cite[Th.4.2.1]{DemZei}) $\varphi_{\bf t}$ is also good a rate function on $\mathbb{R}$, in particular, it is lower semicontinuous.

Let us emphasize that if we know that $(\psi^1)^\ast$ is a good rate function then we can prove lower-semicontinuity of $\varphi_{\bf t}$.
\end{exa}
\begin{rem}
\label{uw}
Let assume now that $I$ is a finite set and $Y_i=t_iX_i$ then $X_{\bf t}=\sum_{i\in I} Y_i$. Note that 
$$
\psi_{\bf t}(s)=\sum_{i\in I}\psi_{Y_i}(s).
$$
The functions $\psi_{Y_i}$, $i\in I$, are convex, lower semicontinuous, not identically equals $+\infty$ and continuous at $0$ that is ones satisfy the assumption of Theorem 1 \cite{HiU}. The convex conjugate of their sum is given by the infimal convolution 
\begin{equation}
\label{infconv}
\psi_{\bf t}^\ast(\alpha)=\min_{\sum\alpha_i=\alpha}\sum_{i\in I}\psi_{Y_i}^\ast(\alpha_i).
\end{equation}
Since $\psi_{Y_i}(s)=\psi_{t_iX_i}(s)=\psi_{X_i}(t_is)$, by scaling property (\ref{jedn}), one has 
$$
\psi_{Y_i}^\ast(\alpha_i)=\psi_{t_iX_i}^\ast(\alpha_i)=\psi_{X_i}^\ast\Big(\frac{\alpha_i}{t_i}\Big).
$$
Substituting $\alpha_i=b_it_i$ into (\ref{infconv}) we obtain
$$
\psi^\ast_{\bf t}(\alpha)=\min_{\sum b_it_i=\alpha}\sum_{i\in I}\psi_{X_i}^\ast(b_i).
$$
It means that Theorem \ref{mtw} one can treat as  the special case of a generalization of the infimal convolution to the situation of infinite numbers of summands.

\end{rem}

\end{document}